\documentclass[12pt]{article}
\usepackage{geometry,amssymb,amsmath,enumerate, tikz, hyperref, amsthm, amsfonts}

\newtheorem{theorem}{Theorem}[section]

\newtheorem{lemma}[theorem]{Lemma}
\newtheorem{definition}[theorem]{Definition}

\newcommand{\ex}{\mathrm{ex}}
\newcommand{\forb}{\mathrm{forb}}

\newcommand{\PP}{\mathbb{P}}
\newcommand{\EE}{\mathbb{E}}

\newcommand{\B}{\mathcal{B}}

\newcommand{\D}{\mathcal{D}}

\newcommand{\F}{\mathcal{F}}
\newcommand{\G}{\mathcal{G}}
\newcommand{\I}{\mathcal{I}}

\newcommand{\K}{\mathcal{K}}

\newcommand{\Hc}{\mathcal{H}}
\newcommand{\Lc}{\mathcal{L}}

\newcommand{\Pc}{\mathcal{P}}

\newcommand{\ve}{\varepsilon}

\newcommand{\floor}[1]{\left\lfloor #1\right\rfloor}
\newcommand{\ceil}[1]{\left\lceil #1\right\rceil}
\newcommand{\La}{\mathrm{La}}

\newtheorem{prob}{Problem}

\title{On the number of families avoiding a subposet}

 \author{Tao Jiang \footnote{Dept. of Mathematics, Miami University, Oxford, OH 45056, USA, {\tt jiangt@miamioh.edu}.  }\and Sean Longbrake \footnote{Dept. of Mathematics, Emory University,  Atlanta, GA 30322, USA, {\tt sean.longbrake@emory.edu}} \and Liana Yepremyan \footnote{Dept. of Mathematics, Emory University,  Atlanta, GA 30322, USA {\tt lyeprem@emory.edu}.  Research is supported by the National Science Foundation grant 2247013: Forbidden and Colored Subgraphs.}}
\date{\today}

\begin{document}

\maketitle

\begin{abstract}
In this paper we show that for any poset $P$ that is not an antichain, the number of induced $P$-free families in  the Boolean lattice $2^{[n]}$ is at most $ 2^{O(\La^*(n,P))}$, where $\La^*(n,P)$ denotes the the largest size of an induced $P$-free subfamily of  $2^{[n]}$. We also obtain related supersaturation results. 
\end{abstract}

\section{Introduction}

Recall that the Boolean lattice $2^{[n]}$ is the power set of $[n] = \{1, \dots,  n\}$ partially ordered under inclusion. The \emph{height} of $P$ is the size of the largest chain in $P$, and the \emph{width} of
$P$ is the size of the largest antichain in $P$. The \emph{dimension} of the poset $P$ is the smallest $d$ such that there exist $d$ linear orderings $\varphi_1, \dots,  \varphi_d: P \rightarrow \{1, ..., |P|\}$ for which for $x, y \in P$, we have $x<_{P}y$ if and only if $\pi_i(x) < \pi_j(y)$ for $i = 1, \dots,  d$. Let $P$, $Q$ be two finite posets, that is, they are finite sets equipped with partial orders $<_{P}$ and $<_Q$. An {\it induced poset homomorphism} is a function $f: P\to Q$ such that $f(x) <_Q f(y)$ if and only if $x <_P y$. We say that a poset $Q$ contains an induced copy of another poset $P$ if there is an injective induced poset homomorphism from $P$ to $Q$. If $Q$ does not contain an induced copy of $P$, we say that $Q$ is  {\it induced $P$-free}. 

Given a poset $P$ and an integer $n$, we define $\La(n,P)$ to be the largest size of a $P$-free subfamily of $\B_n$  and $\La^*(n,P)$ the largest size of an induced $P$-free subfamily of the Boolean lattice, $2^{[n]}$. In this paper we will study questions related to $\La^*(n,P)$. Regarding $\La(n,P)$,  we refer the reader to the survey of Griggs and Li~\cite{GLi}. For a small sample of more recent work on $\La(n,P)$ see, for instance, ~\cite{BGW, BTW, collaresmorris, EIL, GNPV}. Regarding $\textrm{La}^*(n,P)$, the following 
theorem was established by Methuku and P\'alv\"olgyi~\cite{methuku2017forbidden}  and subsequently, Tomon~\cite{tomon2019forbidden}. In \cite{methuku2017forbidden}, $C_P$ is typically exponential in $|P|$. For posets $P$ of constant heights, Tomon \cite{tomon2019forbidden} improved this to $C_P=|P|^{c_h},$  where $c_h$ only depends on the height $h$ of $P$.

\begin{theorem}[\cite{methuku2017forbidden, tomon2019forbidden}]
For every poset $P$, there exists a constant $C_P$ such that $$\La^*(n,P)\leq C_P\binom{n}{\lfloor n/2\rfloor}.$$
\end{theorem}

In this paper we consider two related classical questions, one is the \emph{counting question}, the number of induced $P$-free families in the Boolean lattice $[2]^n$ for a general poset $P$, denoted by $\forb^*(n,P)$, and the second is the \emph{supersaturation} question which asks how many  induced copies of $P$ we get in a family $\mathcal{F}\subseteq [2]^n$  that has more than $\La^*(n,P)$ members.

 The study of the number of $H$-free graphs for a fixed graph  or hypergraph $H$, denoted by $\forb(n,H)$ is a very well developed area. For non-bipartite graphs $H$, the seminal works of Erd\H{o}s, Kleitman, and Rothschild \cite{EKR-Rome} and 
Erd\H{o}s, Frankl, and R\"odl~\cite{EFR} established $\forb(n,H)=2^{(1+o(1))\ex(n,H)}$. This was extended by Nagle, R\"odl, and Schacht \cite{NRS} to non-$k$-partite $k$-uniform hypergraphs $H$. The study of $\forb(n,H)$ for  bipartite graphs and $k$-partite $k$-uniform hypergraphs is much more difficult. There have been quite a lot of breakthroughs (for instance \cite{MS, FMS} ) in recent years due to the development of the powerful container method
\cite{BMS, ST}. Using the container method, Morris and Saxton \cite{MS} showed $\forb(n,C_{2\ell})=2^{O(n^{1 + 1/\ell}) }$, and Ferber, McKinley, and Samotij \cite{FMS} proved far-reaching results, showing that for all  $k$-partite $k$-unform hypergraphs $H$ satisfying $\ex(n,H)\geq \varepsilon n^{k-\frac{1}{m_k(H)}+\varepsilon}$, for some $\varepsilon>0$, $\forb(n,H)=2^{O(\ex(n,H))}$ holds, where $m_k(H)=\max_{F\subset H} \frac{e(F)-1}{v(F)-k}$. In the first part of this paper we consider the analogous question for general posets $P$. Note that the same question for the number of $P$-free subsets of $2^{[n]}$, instead of \emph{induced} $P$-free subsets, has an affirmative answer since any $P$-free family is free of chains of size $|P|$, where $|P|$ is the number of elements in $P$, and by results of Collares and Morris~\cite{collaresmorris}, this implies $\forb(n,P)$ is at most  $2^{(1+o(1))(|P| - 1)\binom{n}{n/2}}$.  For induced setting, recently the authors of this paper together with Spiro~\cite{JLSY} showed that 
 for posets whose Hasse diagram is a tree of height $h$, $\forb^*(n,P)\leq 2^{(h - 1 + o(1))\binom{n}{n / 2}}.$ This extended previous work of Patk\'os and Treglown~\cite{PT} and Gerbner, Nagy, Patk\'os, and Vizer~\cite{GNPV} who obtained similar results for special subclasses of tree posets.  Our first result, stated below, is an upper bound on $\forb^*(n,P)$ for general $P$, in the spirit of the aforementioned results of Ferber, McKinley, and Samotij~\cite{FMS}. 


\begin{theorem}\label{thm:succinctmain}
    For any poset $P$ that is not an antichain,
    \[\forb^*(n,P)\leq 2^{O(\La^*(n,P))}.\]
\end{theorem}

 Theorem~\ref{thm:succinctmain} follows from the following  technical theorem, combined with the results of Methuku and P\'alv\"olgyi \cite{methuku2017forbidden} and Tomon~\cite{tomon2019forbidden} on $\La^*(n,P)$.
 
\begin{theorem}\label{thm:main}
Let $P$ be a poset and $C_P$ be some constant such that for all $n \geq n_0$, $\La^*(n, P) \leq C_P \binom{n}{n / 2}$. Then, there exists $n_1$ such that for all $n \geq n_1$,  the number of induced $P$-free  families in the Boolean lattice $[2]^n$ is at most
$$\forb^*(n,P) \leq \exp\left(490\frac{|P|^2}{\log(|P|)} C_p \binom{n}{ n / 2}\right).$$ 
\end{theorem}

It would be desirable to improve the bound to  $\forb(n, P) \leq  2^{ (C_P +o(1)) \binom{n}{n/2}}$. Under the condition that $\lim_{n \rightarrow \infty} \frac{\La^*(n, P)}{\binom{n}{n/2}}$ exists, this would be best possible, in particular implying that $\forb(n, P) = 2^{(1 + o(1))\La^*(n, P)}$. 


Now we move to the second focus of our paper, supersaturation. As a main ingredient to establishing  Theorem~\ref{thm:main}, we derive the following \emph{balanced} supersaturation bound. Namely, if $|\F| = t C_P \binom{n}{n/2}$ with $t \geq e$, then the number of copies of $P$ in $\F$ is at least $\Omega(e^{t \log(|P|)/(e|P|)} |\F|)$ (see  Theorem~\ref{mainhammer}). To get stronger supersaturation bounds, we exploit the connection between the extremal numbers of posets in grids, and $\La^*(n,P)$, as done in~\cite{methuku2017forbidden, tomon2019forbidden}. Let us now define these extremal numbers.

\begin{definition}
    Let $k_1, \dots k_d$ be positive integers. The Cartesian product $[k_1] \times \dots \times [k_d]$ has a natural partial ordering $\preceq$, with $(a_1, \dots a_d) \preceq (b_1, \dots b_d)$ if for all $i$, $a_i \leq b_i$. We shall call this poset structure a $d$-dimensional grid. The sides of the grid are $k_1, \dots k_d$. If $k_1 = k_2 = \dots = k_d = k$, we will write $[k]^d$ for
    $[k_1] \times\dots \times [k_d]$. 
\end{definition}

For integers $n, d$, let $\ex^*(n, d, P)$ be the maximum size of a subfamily of $[n]^d$ that is induced $P$-free.  The following bound on $\ex^*(n,d,P) $ is implied by the main result of Klazar and Marcus~\cite{KM}, also reiterated in the paper of Methuku and P\'alv\"{o}lgyi ~\cite{methuku2017forbidden}. (See \cite{tomon2019forbidden} for detailed discussion.)

\begin{theorem}[Theorem 1.3 in~\cite{methuku2017forbidden}]\label{thm:KM}Given a poset $P$,  let $d$ be its dimension, then for any $n$, the following is true:
$$\ex^*(n,d,P)\leq  2^{O_d(|P|\log{|P|})}n^{d-1}.$$
\end{theorem}

The leading coefficient in Theorem~\ref{thm:KM} was improved to $2^{O_d(|P|)}$ by Geneson and Tian in \cite{geneson2017extremal}. For posets of constant height, 
Tomon~\cite{tomon2019forbidden} substantially improved  the bound in Theorem \ref{thm:KM}
as follows.

\begin{theorem}[Theorem 12~in~\cite{tomon2019forbidden}]\label{thm:Tomon} There exists a constant $c_h$ depending only on $h$ such that the following holds for $P$ of height at most $h$. Let $n$ be a positive integer, $d = 2|P|$. Then 

$$\ex^*(n,d,P)\leq |P|^{c_h}n^{d-1}.$$
\end{theorem}


Building on the connection to extremal problem in the grids, we obtain the following superstaturation result, which gives us an improvement of a polynomial factor over
the bound obtained in Theorem \ref{mainhammer}.

\begin{theorem}\label{thm:supersaturation}
    There exists an absolute constant $K$ such that the following holds for every poset $P$. Suppose there exists constants $C_P, d$ such that $\ex^*(n, d, P) \leq C_p n^{d - 1}$. Then for every positive integer $t \geq 1$, there exists a $c_{t, d}>0$, such that the following holds. If $n$ is sufficiently large and $\F \subseteq [2]^n$, such that $|\F| \geq (t + K \sqrt{d} C_p + \varepsilon) \binom{n }{n/2}$, the number of induced copies of $P$ contained in $\F$ is at least  $$ c_{t, d} \ve n^{\floor{\frac{t}{K\sqrt{d}}}} \binom{n}{ n/2}.$$ 
\end{theorem}

For our purposes, when $|\F|$ is large enough that we can use the bound of Theorem~\ref{thm:KM} on $\ex^*(n, d, P)$ with $d$ the dimension $P$ in Theorem~\ref{thm:supersaturation}. In this case, the result we derive is stronger than applying Theorem~\ref{thm:supersaturation}, with the result of Theorem~\ref{thm:Tomon}. However, for many posets, Theorem~\ref{thm:Tomon} allows one to derive supersaturation results for smaller families than using Theorem~\ref{thm:KM}.


\section{Preliminary Lemmas}
For a subset $P$ of $(Q, \preceq)$ with $Q$ a distributive lattice, we let $\cup_Q P$ be the minimum element $B$ of $Q$ satisfying for every $A \in P$, $A \preceq B$. Similarly, we let $\cap_Q P$ be the maximum element $B$ of $Q$ such that for every $A \in P$, $A \succeq B$. If the host lattice is clear, we drop the subscript $Q$ from the notation. Given $A \preceq  B$ in $Q$, we let $|B - A| = |C| - 1$, with  $C$ being the longest chain in $Q$ such that $\cup C = B$ and $\cap C =A$.  
\begin{definition}
For $S \subseteq 2^{[n]}$, let $d(S) = | \cup S - \cap S|.$
\end{definition}
\begin{definition}
For arbitrary poset $P$, let 
\[d^*(P) = \min\{ d(\varphi(P)): \varphi \mbox{ is an induced embedding of $P$ into $2^{[n]}$ for some $n$}\}. \footnote{Note this is well defined by the well-ordering property of the natural numbers. }\]
\end{definition}
\begin{definition} \label{def:mup}
    Let $\mu(P) = \min_{D \subseteq P, |D| \geq 2} \frac{d^*(D)}{|D| - 1}$. 
\end{definition}
Observe $\mu(2^{[n]}) = \frac{n}{2^n - 1}$,  $\mu(C_n) =1$, and $\frac{\log|P|}{|P|} \leq \mu(P) \leq 1$. 

Let $A\subseteq B\subseteq [n]$. Let $2^{[n]}[A,B]=\{C\subseteq [n]: A\subseteq C\subseteq B\}$ and call it the {\it interval sublattice} of $2^{[n]}$ {\it spanned by} $A,B$.
It is easy to see that for any $1\leq m\leq n$ there are $\binom{n}{m}2^{n-m}$ interval sublattices of $2^{[n]}$ with dimension $m$. Indeed, if $S_A,S_B$ are the $(0,1)$-indicator vectors of length $n$ for $A,B\subseteq [n]$ then 
$2^{[n]}[A,B]$ is determined by specifying the $m$ coordinates $S_A, S_B$ differ in and
the value of each of the remaining $n-m$ coordinates.

\begin{lemma}\label{chaincontainedinlattice}
Let $S\subseteq 2^{[n]}$.
Let $W$ be an interval sublattice of $2^{[n]}$ of dimension $m$ chosen uniformly at random. 
Then
$$\PP(S \subseteq W) = \frac{\binom{m}{d(S)}}{\binom{n}{d(S)} 2^{n - m}}.$$

In particular, for any $C \in 2^{[n]}$, 

$$\PP(C \in W) = \frac{1}{2^{n - m}}.$$
\end{lemma}
\begin{proof}
By our assumption $W=2^{[n]}[A,B]$ for some $A,B\in 2^{[n]}$ with $|B\setminus A|=m$.
Since $W$ is a sublattice, $W$ contains $S$ if and only if 
$A\subseteq \cap S\subseteq \cup S\subseteq B$.
In particular, we have 
\[A=\cap S\setminus (B\setminus A) \mbox{ and } B=\cup S \cup (B\setminus A).\]

Hence, the pairs $A,B$ are completely determined by $B\setminus A$.
Since $B \setminus A$ contains $\cup S \setminus \cap S$, 
there are $\binom{n - d(S)}{m - d(S)}$ ways to choose $B\setminus A$.
Thus, we have exactly $\binom{n - d(S)}{m - d(S)}$ choices of $W$ that contains $S$. 
So, using the identity $\binom{n}{m}\binom{m}{d(S)}=\binom{n}{d(S)}\binom{n-d(S)}{m-d(S)}$, we have
\[\PP(S\subseteq W)= \frac{\binom{n-d(S)}{m-d(S)}}{\binom{n}{m} 2^{n-m}}=
\frac{\binom{m}{d(S)}}{\binom{n}{d(S)} 2^{n - m}}.\]
\end{proof}

\section{Proof of Main Theorem}

Let $P$ be a given poset. First, we show that when $\F\subseteq 2^{[n]}$ is sufficiently dense, there exists a dense balanced collection of induced copies of $P$ in $\F$.

\begin{theorem}\label{mainhammer}
Let $C_P$ be some constant such that for all $n \geq n_0$, $La^*(n, P) \leq C_P \binom{n}{n / 2}$. Let $t$ be an integer. Then, there exists an $n_1$ such that for all $n \geq n_1$, the following holds for all $k$ satisfying $ 8\leq k \leq \sqrt{\frac{4n}{n_0}}$: Let $\F \subseteq 2^{[n]}$ satisfy that $|\F| = kt C_P \binom{n}{n / 2}$. Then, there exists a collection $\Hc$ of induced copies of $P$ satisfying the following properties: \begin{enumerate}
    \item $|\Hc| \geq \frac{1}{2t} (k/8)^{2 t \mu(P) (|P| - 1)} |\F| $
    \item For every $S \subseteq 2^{[n]},$ $\deg_\Hc(S) \leq 2K_P \cdot (k/8)^{ 2 t \mu(P) (|P| - |S|)} $.
\end{enumerate}
 where $K_P = 2^{|P| + \log(P) + 2}$ and $\mu(P)$ is as defined  
 in Definition \ref{def:mup}.
\end{theorem}
\begin{proof}
 For each $i=0,\dots,t-1$, let $\G_i=\{F\in \F: |F|\equiv i \pmod{t}\}$.
 Then for some $i$, $|\G_i|\geq  kC_P\binom{n}{n/2}$. Fix such an $i$, and let $\F_i$ be a subfamily of exactly size $kC_P\binom{n}{n/2}$ inside $\G_i$. We will build $\Hc$ greedily inside $\F_i$. Note that for all $S \subseteq \F_i$, $d(S) \geq t d^*(S)$ by the definition of $\F_i$. 
 Assume on the contrary that there is a maximal $\Hc$ satisfying condition two such that $|\Hc| <\frac{1}{2t} (k/8)^{2 t \mu(P) (|P| - 1)} |\F|  \leq \frac{1}{2} (k/8)^{2 t \mu(P) (|P| - 1)} |\F_i| $. 

 We call a set $S \subseteq 2^{[n]}$ dangerous if $\deg_{\Hc}(S) \geq K_P \cdot  (k/8)^{ t \mu(P) (|P| - |S|)}$. Let $\D_s$ be the set of all dangerous sets $S$ with $|S| = s$. 
 Observe that
 
 $$|\D_s| \leq K^{-1}_P 2^{|P|} (k/8)^{ 2 t \mu(P) (s- 1) }|\F_i|,$$ as $$|\D_s| \cdot K_P \cdot  (k/8)^{ 2t \mu(P) (|P| - |S|)} \leq |\Hc| \cdot 2^{|P|} \leq \frac{1}{2} (k/8)^{2 t \mu(P) (|P| - 1)} |\F_i| \cdot 2^{|P|}.$$

Let $W$ be an interval sublattice of $2^{[n]}$ of dimension $m = \floor{\frac{4n}{k^2}}$ chosen uniformly at random. By Lemma~\ref{chaincontainedinlattice}, and since $|\F_i|\geq kC_P\binom{n}{n/2}$, we have
    
 \begin{align*}
 \EE[|W \cap \F_i|] \geq 2^{m - n} k C_P \binom{n}{n / 2}
    \geq  (2 + o(1)) \sqrt{\frac{2}{\pi n}} \sqrt{\frac{n}{m}} C_P 2^m
    \geq (2 +o(1)) C_P \binom{m}{m / 2}.
\end{align*}
Let $X_s$ denote the number of members of $\D_s$ that are contained in $W$. Then
\begin{align*}
    \EE(X_s) &= \sum_{S \in \D_s} \frac{\binom{m}{d(S)}}{\binom{n}{d(S)} 2^{n - m}} 
     = \sum_{S \in \D_s} \frac{(m)_{d(S)}}{(n)_{d(S)} 2^{n - m}}
    \leq \sum_{S \in \D_s} \frac{(m)_{t(s - 1)\mu(P)}}{(n)_{t(s - 1)\mu(P)} 2^{n - m}}\\
    &\leq \frac{2^{|P|}}{K_P} \left(\frac{k}{8}\right)^{2t \mu(P) (s - 1)} |\F_i|  \cdot  \left( \frac{2m}{n} \right) ^{t(s - 1) \mu (P)} 2^{m - n  }\\
    &\leq \frac{2^{|P|+1}}{K_P} \left(\frac{k}{8}\right)^{2t \mu(P) (s - 1)} C_p \sqrt{\frac{n}{m}} \binom{n}{n/2}  \cdot  \left( \frac{8}{k^2} \right) ^{t (s - 1) \mu (P)} 2^{m - n  } \\
    &\leq \frac{1}{2|P|} C_P \binom{m}{ m / 2},
\end{align*}
 where we used $K_P = 2^{|P| + \log(P) + 2}$.

Let $\Lc$ denote the collection of all induced copies of $P$ in $\F_i$ that
do not contain any dangerous set. We bound $|\Lc|$ as follows.
From each interval sublattice $W$ contained in $2^{[n]}$ of dimension $m$, we remove one element from every dangerous set contained in $W$ to obtain $W'$. Since $|W'\cap \F_i|\geq |W \cap \F_i| - \sum_{s = 1}^{|P|} X_s$, we can find $|W\cap \F_i| - \sum_{s = 1}^{|P|} X_s - C_P \binom{m}{ m /2}$  many copies of $P$ not containing any dangerous set $S$ inside $W'$. Summing over all $W$ contained in $2^{[n]}$ of dimension $m$, we find at least 

$$2^{n} \binom{n}{m} \left( \EE[|W \cap \F_i| - \sum_{s \in [|P|]} {X_s}] - C_P \binom{m}{m /2} \right) \geq 2^{n - 1} C_P \binom{n}{m} \binom{m}{ m/2}$$ pairs of a $m$-dimensional lattice and a copy of $P$ not containing a dangerous set inside this lattice. On the other hand by Lemma~\ref{chaincontainedinlattice}, each copy of $P$ appears in no more than $2^{m} \binom{n}{m} \left(\frac{2m}{n} \right)^{td^*(P)}$ $m$-dimensional lattices. Therefore, 

\begin{align*}
   |\Lc|\geq \frac{1}{2} C_P \binom{m}{ m / 2} 2^{n - m} \left(\frac{n}{2m} \right)^{td(P)} &= \frac{1}{2} C_P \sqrt{\frac{2}{\pi m}}  2^{n - m} \cdot 2^m \cdot \left(\frac{k^2}{8} \right)^{td(P)} \\
    &= \frac{1}{2} k^{2t \mu(P) \cdot (|P| - 1) }C_p k \sqrt{\frac{2}{\pi n}} \cdot 2^n \\
    &= \frac{1}{2} (k/8)^{2t \mu(P) \cdot (|P| - 1)} \cdot |\F_i| >|\Hc|
\end{align*}
So, there exists $F\in \Lc \setminus \Hc$. Note that $F$ can be added to $\Hc$ without violating condition two by the definitions of $\Lc$ and dangerous sets. But this contradicts the maximality of $\Hc$. 
\end{proof}
\begin{theorem}[Container Lemma]\label{containerlemma}[\cite{balogh2019efficient}]
Let $r$ be a positive integer and $\Hc$ a nonempty $r$-uniform hypergraph such that $\tau \in (0, 1)$ and $A > 0$ are such that $\tau \cdot  v(\Hc) \geq 10^8 r^6 A$ and for every $s \in \{2, \dots, r\}$, 
$$\Delta_t(\Hc) \leq A \cdot \left( \frac{\tau}{10^6 r^5}\right)^{ s- 1} \cdot \frac{e(\Hc)}{v(\Hc)}.$$

Then, there is a family $\mathcal{C} \subseteq \binom{V(\Hc)}{ \leq \tau \cdot v(\Hc)}$ and $f: \mathcal{C} \rightarrow \mathcal{P}(V(\Hc))$ and $g: \mathcal{I}(\Hc) \rightarrow \mathcal{C}$ such that for every $I \in \I(\Hc)$, 

$$g(I) \subseteq I \subseteq g(I) \cup f(g(I)) \text{ and } |f(g(I)) \leq (1 - \delta) \cdot v(\K), $$

where $\delta = (10^3r^4 K)^{-1}$. Moreover, if $g(I) \subseteq I'$ and $g(I') \subseteq I$ for some $I, I' \in \mathcal{I}(\K)$, then $g(I) = g(I')$. 
\end{theorem}

We will now prove the following result, which immediately implies Theorem~\ref{thm:main} using $\mu(P) \geq \frac{\log(|P|)}{|P|}$ and the result being trivial for $|P| =1$.
\begin{theorem}
Let $P$ be a poset on at least two elements and $C_P$ be some constant such that for all $n \geq n_0$, $\La^*(n, P) \leq C_P \binom{n}{n / 2}$. Then, there exists $n_1$ such that for all $n \geq n_1$,  the number of induced $P$-free  families in the Boolean lattice $[2]^n$ is at most
$$\forb^*(n,P) \leq \exp\left(490\frac{|P|}{\mu(P)} C_p \binom{n}{ n / 2}\right).$$     
\end{theorem}

\begin{proof}
    Let $t =30 |P|/\mu(P)$.
    We apply the container lemma iteratively, so to represent this, we form a rooted tree $T$.  At step $0$, the tree is the singleton $\F= 2^{[n]}$. At step $i$, we examine each leaf $\F$ of the tree, and possibly add new children to $\F$ to grow $T$. There will be three cases for us. Write $|\F| = ktC_P \binom{n}{n/2}$. One case will be if $k < 16$, one if $k \geq  \sqrt{\frac{4n}{ n_0}}$, and one if $16 \leq k < \sqrt{\frac{4n}{ n_0}}$. 
    
    \textbf{Case 1:} If $k < 16$, we are done with this leaf and will do nothing.

    \textbf{Case 2:} If $k \geq \sqrt{\frac{4n}{n_0}}$, take $\F' \subseteq \F$, a family of size $\sqrt{\frac{4n}{ n_0}}tC_P \binom{n}{n/2}$. We then apply Theorem~\ref{mainhammer} to $\F'$ to form a collection $\Pc$ of induced copies of $P$ in $\F'$. Let $\Hc$ be a hypergraph on vertex set $\F$ with edges corresponding to copies of $P$ in $\Pc$. $\Hc$ will satisfy Theorem~\ref{containerlemma} with $A = \sqrt{n_0}2^{3|P| + 11}$, $\tau^{-1} = |P|^{-5}(\frac{n}{256n_0})^{t \mu(P)}$, and $\delta = (\sqrt{n_0}10^72^{6|P|})^{-1}$. Then, Theorem~\ref{containerlemma} gives a collection of families $\F_1, \dots \F_m$ contained inside $\F$ such that $m \leq  \exp((\sqrt{n})^{-1}\binom{n}{n/ 2})$ for $n$ sufficiently large, and that for every $P$-free subfamily of $\F$ is contained in $\F_i$ for some $i$. We add these $\F_i$ as the children of $\F$ to grow the tree $T$. 
     
    \textbf{Case 3:} If $16 \leq k < \sqrt{\frac{4n}{ n_0}}$, we apply Theorem~\ref{mainhammer} to $\F$ to form a collection $\Pc$ of induced copies of $\Pc$ in $\F$. Let $\Hc$ be a hypergraph on vertex set $\F$ with edges corresponding to copies of $P$ in $\Pc$. $\Hc$ will satisfy Theorem $3.2$ with $A = 2^{3|P| + 9}$, $\tau^{-1} = 10^{-6}|P|^{-5}(k/8)^{2t \mu(P)}$, and $\delta = (10^3 2^{5|P| + 9})^{-1}$. Then, Theorem~\ref{containerlemma} gives a collection of families $\F_1, \dots \F_m$ contained inside $\F$ such that 
    \begin{align*}
        m &\leq \exp\left(2\log(10^{-6}|P|^{ - 5} (k/8)^{2t \mu(P)}) 10^6 |P|^5 (k/8)^{ - 2t \mu(P)} ktC_P \binom{n}{n/ 2}\right) \\
        &\leq \exp\left(2^{27} C_P |P|^{9} (k/8)^{-2t \mu (P)} k^2 \binom{n}{n/2}\right)    \end{align*} where we used that $t \leq 30 |P|^2$. We also know that for every $P$-free subset of $\F$ is contained in $\F_i$ for some $i$. We add these $\F_i$ as the children of $\F$ to grow the tree $T$.

    Consider the tree $T$ at the end of the process. Observe that $$\forb^*(n, P) \leq \sum_{\F \textrm{ is a leaf in } T} 2^{|\F|}.$$

    Observe that the number of leaves in $T$ is no more than \begin{align*}
    &\exp\left(\frac{8\log(16n_0) \sqrt{n_0} 10^7 2^{6|P|})}{\sqrt{n}}\binom{n}{n/2} + \sum_{i = 0}^{\infty} |P|^92^{33}C_P (2 )^{-2t \mu(P)}(1 - (2^{5|P|+ 19} )^{-1})^{it \mu(P)} \binom{n}{n/2}\right) \\&\leq \exp\left(\frac{1}{2}\binom{n}{n/2} + 2^{9|P| + 54}C_P2^{-60 |P|} \right) \leq \exp\left(C_P \binom{n}{n/2} \right),
    \end{align*}

where in the first inequality we used that $n$ is sufficiently large and the second that $|P| \geq 2$. 
Since for each leaf $2^{|\F|} \leq 2^{\frac{480 |P|}{\mu(P)} C_P \binom{n}{n/2}}$, we have $\forb^*(n,P) \leq \exp\left(490\frac{|P|}{\mu(P)} C_p \binom{n}{ n / 2}\right).$

    
\end{proof}

\section{Supersaturation}

\begin{theorem}\label{supersaturation2}
    There exists an absolute constant $K$ such that the following holds for every poset $P$. Suppose there exists a constant $C_P$ such that $\ex(n, d, P) \leq C_p n^{d - 1}$. Then for every positive integer $t \geq 1$, there exists a $c_{t, d}>0$, such that the following holds. If $n$ is sufficiently large and $\F \subseteq 2^{[n]}$, such that $|\F| \geq (t + K \sqrt{d} C_p + \varepsilon) \binom{n }{n/2}$, the number of induced copies of $P$ is at least  $$ c_{t, d} \ve n^{\floor{\frac{t}{K\sqrt{d}}}} \binom{n}{ n/2}.$$ 
\end{theorem}

Note that if $P$ has a minimum and maximum element, then the number of induced copies of $P$ in the middle $t$ levels of $2^{[n]}$ is $\Theta(n^{ t - 1}\binom{n}{n/2})$. 

Our proof technique actually produces more \emph{structured} copies. We say that a copy of $P$ in a poset $Q$ is $t$-gapped, if $|\cup P -  \cap P| \geq t$.  We will use $\ex_{t}^*(n, d, P)$ to be the largest subset of $[n]^d$  without an induced copy of $P$ which is $t$-gapped with respect to $[n]^d$. It turns out that for a fixed poset $P$ these two extremal numbers, $\ex_{t}^*(n, d, P)$ and $\ex(n, d, P)$,  are not too far from each other, as the next lemma shows. Our technique of the proof of Theorem~\ref{supersaturation2} exploits this fact.

\begin{lemma}\label{gappedturanbound}
Let $P$ be a poset on at least two elements. Then, for all integers $t\geq 1$, 
    $$\ex_t^*(n, d, P) \leq  \ex(n, d, P) + tn^{d - 1}.$$
\end{lemma}
\begin{proof}
    Given $T = (x_1, \dots x_d) \in [n]^d$, we let $\pi_i(T) = x_i$. Fix a subset $\F$ of $[n]^d$ of size at least $\ex(n, d, P) + tn^{d - 1}$.
    Let $\F'$ be the set of $(x_1, \dots x_d)$ satisfying that the number of $y \in [n]$ such that $y > x_1$ and $(y, x_2, \dots x_d) \in \F$ is strictly less than $t$. Observe that $|\F'| \leq tn^{d - 1}$. 
    Let $\F'' = \F \setminus \F'$. Since $|\F''| \geq \ex(n, d, P)$, $\F''$ contains a copy $\varphi(P)$ of $P$. We examine the projection of $\varphi(P)$ onto the first side of the grid, and let $x$ be the maximum element among the projections of $\varphi(P)$ onto the first side of the grid. 
    Let $A$ be a maximal element among elements of $\varphi(P)$ whose projection onto the first side of the grid equals $x$.
    
    Since $A \in \F''$, there exists a $B \in \F'$ which agrees on all coordinates of $A$ except the first, where $\pi_1(B) \geq \pi_1(A) + t$. Let $P^* = (\varphi(P) \setminus \{A\}) \cup \{B\}$. We have that $|\cup P^* - \cap P^*| \geq t$. Next, we prove that $P^* \cong P$, which would complete the proof.  
    Consider any $C \in \varphi(P)$. By maximality of $A$, either $C \preceq A$ or $C \not \preceq A$. If $C \preceq A$, then $C \preceq B$. So, we may assume $C \not \preceq A$. By choice of $x$, if $C \not \preceq A$, then there exists a coordinate $i \neq 1$ such that $\pi_i(C) > \pi_i(A)$. So, we have that $\pi_i(C) > \pi_i(B)$, and thus $C \not \preceq B$. This shows that $P^* \cong P$. 
\end{proof}

\begin{lemma}\label{standardaveraging}
    Suppose $G$ is $[n_1] \times [n_2] \times \dots \times [n_d]$, with $n_1 \leq n_2 \leq \dots \leq n_d$. Furthermore, suppose $P$ is a poset such that $\ex(n,  d, P) \leq C_Pn^{d - 1}$. Then, the largest $t$-gapped induced $P$-free subfamily of $G$ is at most $\frac{C_P + t}{n_1}|G|$. 
\end{lemma}
\begin{proof}
Fix a $t$-gapped $P$-free subfamily $\F$ of $G$. 
Pick a subset $S_i$ of size $n_1$ from $[n_i]$ uniformly at random. There are $\binom{n_i}{n_1}$ choices for each $S_i$. For a given set $A \in G$, the probability $A$ is in $S_1 \times S_2 \times \dots \times S_d$ is selected is $\prod_{i = 1}^d \frac{n_1}{n_i}$. 

Thus, in expectation, the family $\F' = \F \cap S_1 \times \dots \times S_d$ has size $|\F| \cdot \prod_{i = 1}^d \frac{n_1}{n_i}$. By assumption and Lemma~\ref{gappedturanbound}, we have that $|\F'| \leq (C_p + t)n^{d - 1}_1$. Combining, it follows that $|\F| \leq \frac{C_p + t}{n_1} |G|$. 
\end{proof}

We will use the following lemma of Tomon~\cite{tomon2019forbidden}. 

\begin{lemma}[Corollary 8 in \cite{tomon2019forbidden}]\label{gridpartition}
    Let $n, d$ be positive integers such that $n \geq d$. Let $m_1, \dots m_d$ be such that $m_1 +  \dots + m_d = n$ and $m_1, \dots, m_d \in \{ \floor{n/d}, \ceil{n / d}\}$. Then $2^{[n]}$ can be partitioned into $d$-dimensional grids $G_1, \dots G_s$ such that each side is a chain of length at least $c \sqrt{n / d})$ and at most $2c \sqrt{n / d}$.
\end{lemma}

\begin{lemma}\label{chaincontainedingrid}
   Let $G_1, \dots G_s$ be any partition of $2^{[n]}$ into $d$-dimensional grids.
    Let $\pi$ be a random permutation of $[n]$ and for all $i$ let $G_i' = \pi(G_i)$. Suppose $A \subseteq B$. Then $\PP(B \in G_i' | A \in G_i') \leq  \binom{|B - A| + d - 1}{d - 1} \frac{|B - A|!(n - |B|)!}{(n - |A|)!}$.
\end{lemma}
\begin{proof}
Suppose $A \in G_i' = S_1 \times S_2 \times \dots \times S_d$ with each $S_i$ a chain. For $B \in G_i'$, we must have that for all $i$, $B \cap S_i$ is an initial segment of $S_i$ and $B = \bigcup_{i = 1}^d B \cap S_i$. Since we know $A \in G_i'$, we have that $A \cap S_i$ is an initial segment of $S_i$ and $A = \bigcup_{i = 1}^d A \cap S_i$. So, we have that $(B \setminus A) \cap (S_i \setminus A)$ must be an initial segment of $S_i \setminus A$ and $B \setminus A = \bigcup_{i = 1}^d ( B \setminus A) \cap (S_i \setminus A)$. There are at most $\binom{|B \setminus A| + d - 1}{ d- 1} |B \setminus A|!$ ways of distributing the elements of $B \setminus A$ satisfying these two conditions. Any permutation of the remaining $n - |B|$ will still leaves us with $B \in G_i'$. Dividing by the total $(n - |A|)!$ permutations gives the desired bound. 
\end{proof}

\begin{proof}[Proof of Theorem~\ref{supersaturation2}]
    Let $P$ be some poset and let $d$ be a positive integer such that $\ex(n, d, P) \leq C_Pn^{d - 1}$ for some constant $C_P$. Let $K = \frac{\sqrt{\pi}}{c}$, with $c$ given by Lemma~\ref{gridpartition}.  Fix an 
     $\F \subseteq 2^{[n]}$, such that $|\F| \geq (t + K \sqrt{d} C_p + \varepsilon) \binom{n }{n/2}$. 
     Assuming that $n$ is sufficiently large, by a standard reduction, there exists $\F'\subseteq \F$
 such that $|\F'|\geq (t+K \sqrt{d} C_p+\frac{\varepsilon}{2})\binom{n}{n/2}$  and for all $A \in \F'$, $||A| - n/2| \leq 2\sqrt{n \log(n)}$.

Take the partition $G_1, \dots G_s$ of $2^{[n]}$ given by Lemma~\ref{gridpartition}. 
Let $\pi$ be random permutation of $[n]$. For each $i\in [s]$, let $G'_i=\pi(G_i)$.
For ease, let $t' = \floor{\frac{t}{K \sqrt{d}}}$. Note that by assumption and Lemma~\ref{standardaveraging}, the largest $t'$-gapped $P$-free subfamily of $G_i'$ is no more than
$\frac{C_p + t'}{c \sqrt{n / d}}|G_i'|$. By iteratively finding a $t'$-gapped induced copy of $P$ and then removing an edge from it,
one can see the number of $t'$-gapped induced copies of $P$ in $\F' \cap G_i'$ is at least $|\F' \cap G_i'| - \frac{C_p + t'}{c \sqrt{n / d}}|G_i'|$. 

    Therefore the number $\lambda$ of $t'$-gapped induced copies of $P$ in $\bigcup_{i=1}^s G'_i$ is at least
    \[\sum_{i = 1}^s (|\F \cap G'_i| - \frac{C_p + t'}{c \sqrt{n / d}}|G_i'| )
        \geq  |\F'| - \frac{C_p + t'}{c\sqrt{n/ d}}2^n
        \geq |\F'| - (\frac{\sqrt{\pi d}}{c}C_p + t)\binom{n}{n/2},\]
    where we used that for $n$ sufficiently large $\binom{n}{n/2} \geq \frac{1}{\sqrt{\pi n}}2^n$, as $\binom{n}{n/2} = (1 + o(1))\sqrt{\frac{2}{\pi n}}2^n$
    and $t'\leq \frac{t}{K\sqrt{d}}=\frac{c}{\sqrt{\pi d}}t$. Since $|\F'| \geq (t + \frac{\sqrt{\pi d}}{c}C_p + \frac{\ve}{2})\binom{n}{n/2}$, this guarantees us at least $\frac{\ve}{2} \binom{n}{n/2}$ many $t$-gapped copies of $P$. Observe that if $\F' \cap G_i'$ contains a $t'$-gapped copy $P'$ of $P$, then since $G'_i$ is a grid  it must contain both $A:=\cap P'$ and $B:=\cup P'$.  

    By Lemma~\ref{chaincontainedingrid}, the probability $G_i'$ contains $B$ conditioned on $G_i'$ containing $A$ is no more than $\binom{|B| - |A| + d - 1}{d - 1} \frac{(|B| - |A|)!(n - |B|)!}{(n - |A|)!}$. Since all the $G_i'$'s partition $2^{[n]}$, there must be some $G_i'$ which contains $A$. Thus, the probability $P'$ is contained in some $G_i'$ is no more than $\binom{|B| - |A| + d - 1}{d - 1} \frac{(|B| - |A|)!(n - |B|)!}{(n - |A|)!}$. 

    This is monotonically decreasing in $|B| - |A|$ for every fixed $|A|$, since $||B| - |A|| \leq 4 \sqrt{n \log(n)}$. Since $| B \setminus A| \geq t'$ and $\frac{n}{4}\leq |A'|,|B'| \leq \frac{3}{4} n$, we have that this is at most $\binom{t '+ d - 1}{d - 1} t! 8^{t'} n^{-t'}$. 
    Thus, if $\Pc$ denotes the family of induced $t'$-gapped copies in $\F'$, then 
    \[|\Pc|\binom{t '+ d - 1}{d - 1} t! 8^{t'} n^{-t'} \geq \lambda \geq  \frac{\ve}{2} \binom{n}{n/2}. \]
     Thus, $|\Pc|\geq  \frac{\ve }{\binom{t' + d - 1}{d - 1} t'! 8^{t' + 1}}n^{t'} \binom{n}{ n/2}$, as desired.

    \end{proof}

\bibliographystyle{abbrv}

\bibliography{refs.bib}

\end{document}